\documentclass{rmmcart}
\usepackage{amsmath, amsthm, amssymb}   % standard math packages
\usepackage{mathrsfs}                   % for script X
\usepackage{color}                       
\usepackage{enumitem}                 
\usepackage{hyperref}
\usepackage{cite}
 \usepackage{csquotes}
 \newtheorem{theorem}{Theorem}
\newtheorem{lemma}{Lemma}
\def\N{\ensuremath{\mathbb{N}}}
\def\R{\ensuremath{\mathbb{R}}}
\def\X{\ensuremath{\mathscr{X}}}

\def\Ke{\ensuremath{K_e}}
\def\Ue{\ensuremath{U_\text{exit}}}
\def\Uc{\ensuremath{U_\text{crit}}}
\def\Be{\ensuremath{B_e}}
\def\N0{\ensuremath{N_{0,1}}}
\def\x{\ensuremath{\boldsymbol{\xi}}}

     \title{Invariant sets for QMF functions}

     \author{Adam Jonsson}
     \address{Department of Engineering Sciences and Mathematics \\Lule{\aa} University of Technology, 97187 Lule\aa, Sweden}
     \email{adam.jonsson@ltu.se}
     \thanks{This is the accepted version of an original article published by the Rocky Mountain Journal of Mathematics. % \href{https://projecteuclid.org/euclid.rmjm}{Rocky Mountain Journal of Mathematics}.  %This reprint was typeset by the author and differs from the original   in pagination and typographical detail. 
     The author wishes to express gratitude to the referee for valuable suggestions on how to improve the exposition, and to Jeff Steif and  Dick Gundy for critical comments on  an earlier version.}

     \keywords{Scaling functions; Markov processes; Invariant sets}
\begin{document}
     \begin{abstract}
    A quadrature mirror filter (QMF) function can be considered as the transition function for a Markov process on the unit interval. The QMF functions that generate scaling functions for  multiresolution analyses are then distinguished by properties of their  invariant sets. By characterizing these sets,  we answer in the affirmative a question raised by Gundy (Notices Amer. Math. Soc. 57,  1094-1104, 2010).  
     \end{abstract}
     \maketitle
     
\section{Introduction}\label{sec: intro}
The motivation for this paper comes from the study of a class of Markov processes that appear in the construction of scaling functions for  
multiresolution  analyses (MRA). For definitions and background, see  \cite{Coh90,CR90,Dau88,DGH00,Gun00,Kea72,Law90,PSW99} and, in particular, \cite{Gun10}. One way to construct a scaling function is to start with a 1-periodic function $p(\xi),\xi \in \mathbb{R}$, that satisfies
\begin{align}
p(\xi/2)+p(\xi/2+1/2)=1\text{ for every $\xi \in [0,1], \quad p(0)=1$}. 
\label{eq: qmf condition intro}
\end{align}
This condition is known as the  quadrature mirror filter  (QMF) condition. We  reserve the symbol $p$ for nonnegative, continuous 1-periodic functions that satisfy \eqref{eq: qmf condition intro}. We call them \emph{QMF functions}. 

To each $p$ we   associate a Markov process $\xi_0,\xi_1,\xi_2,\dots$  on the  interval $[0,1]$. Given $\xi_0\in [0,1]$, the process evolves according to 
\begin{align}\label{eq: transitions intro}
&\xi_{t+1}=\xi_t/2 \text{ or } \xi_t/2+1/2, \\
&\mathbb{P}_p(\xi_{t+1}=\xi_{t}/2+j/2\vert\xi_t)=p(\xi_t/2+j/2), j=0,1.
\end{align}
Whether or not a given $p$  generates a scaling function % if and only if two conditions are met 
is determined by two conditions (see \eqref{eq: condition (a)} and \eqref{eq: condition (b)} below), the second of which is often trivially satisfied. The first condition says that  
\begin{align}\label{eq: prob condition a}
\mathbb{P}_p(\xi_t \to 0\text{ or }1\vert\xi_0)=1 \text{ for Lebesgue a.e. }\xi_0\in [0,1].
\end{align}
For H\"{o}lder continuous $p$, the left-hand side of the equality in \eqref{eq: prob condition a} is a continuous function of $\xi_0$. In this case,  the equality in \eqref{eq: prob condition a} must hold for every $\xi _0\in [0,1]$ if $p$ generates a scaling function. 
However,  for some $p$ that generate scaling functions, the equality in \eqref{eq: prob condition a}  fails on a set of    measure zero. 

An example of such a $p$ is constructed in   \cite{Gun10}  starting from  $p(\xi)=\cos^2(3\pi\xi)$, a QMF function with $p(1/3)=p(2/3)=1$. That $p$ takes the value one at $\xi=1/3$ and $\xi=2/3$ means that $B=\{1/3,2/3\}$   invariant (i.e., $\mathbb{P}_p(\xi_1  \in B\vert\xi_0)=1$ for every $\xi_0\in B$), so $\mathbb{P}_p(\xi_t  \to 0\text{ or }1\vert\xi_0)=0$ if $\xi_0\in B$.  To allow  paths from initial points in the vicinity  of $B$ to converge to $0$ or $1$, the function $p$  is given sharp cusps at $1/3$ and $2/3$, with corresponding modifications near $1/6$ and $5/6$  to retain the QMF condition. Paths from initial points close to $B$ are still attracted to $B$, but  $B$ is \enquote{inaccessible}: the sequence $\xi_0,\xi_1,\xi_2,\dots$  converges to 0 or 1 with probability   one for every $\xi_0 \in B^c$. Hence,  the equality in \eqref{eq: prob condition a} holds for almost every $\xi_0 \in [0,1]$.

The set $B$ in the above example is closed and invariant under multiplication by 2 (mod 1). Since every  subset of $(0,1)$ with these properties has  measure zero by the ergodicity of the doubling map, we may ask (cf. \cite[p. 1103]{Gun10}): given a closed set $B \subset (0,1)$,  invariant under multiplication by 2 (mod 1), is there a $p$ for which  $B$ is invariant,  where  $\mathbb{P}_p(\xi_t  \to 0\text{ or }1\vert \xi_0)=1$ for   a.e. $\xi_0 \in [0,1]$?

Our objective is to answer this question by establishing the following result, whose proof provides a characterization of those subsets of $(0,1)$ that are invariant with respect to some $p$ (see \eqref{eq: disjointness condition 0} below).   
 \begin{theorem}\label{thm: main result} If $B\subset  (0,1)$ is closed, invariant under multiplication by 2 (mod 1), and invariant for some $p$, then there is a $\tilde{p}$ for which $B$ is invariant, where $\mathbb{P}_{\tilde{p}}(\xi_t  \to 0\text{ or }1\vert\xi_0)=1$ 
 for    a.e.  $\xi_0 \in [0,1]$.
\end{theorem}

The paper is organized as follows. The next   section restates the main question on a space of binary sequences. Having seen the role played by  invariant subsets of the sequence  space, we return to the unit interval in Section \ref{sec: thm1}, where Theorem \ref{thm: main result} is proved. Section \ref{sect: conclude} concludes our study.

\section{The  dynamics of sample paths}
\label{sec: sym dyn of paths}
 
The process  \eqref{eq: transitions intro} is conveniently studied  on the set  of all sequences $\x=(\dots{},x_{-1},x_0)$ of zeros and ones (see \cite{CR90,Gun07,Gun10,GJ10}), viewed as binary representations of points of $[0,1]$. We denote this set by $\X$. The correspondence between $\X$ and $[0,1]$ is given by  $\tau\colon \X\to [0,1]$, where $\tau(\x)=\sum^{\infty}_{j=0} x_{-j} 2^{-(j+1)}$. With the topology induced by the metric
\begin{equation}
\rho(\x,\x')=\begin{cases}
0 \mbox{ if $\x=\x'$},\\
2^{-\min\{|j|:x_j \neq x_j'\}}\text{ if $\x \neq \x'$},\\
\end{cases}
\end{equation}
$\X$ becomes a compact space.  

After composition with $\tau$, a QMF function defines   a continuous $g \colon \X\to [0,1]$ that satisfies 
\begin{align}\label{eq: g}
g((\x,0))+g((\x,1))=1\text{ for all }\x \in X, \quad g(\boldsymbol{0})=g(\boldsymbol{1})=1.
\end{align} 
Here $\boldsymbol{0}=(\dots,0,0)$ and $\boldsymbol{1}=(\dots,1,1)$. If we define 
\begin{align}\label{eq: ast on X}
(\dots,x_{-1},x_0)^\ast=(\dots,x_{-1},1-x_0),
\end{align} 
we can write the first condition in \eqref{eq: g}  as the requirement  that 
\begin{align}\label{eq: g2}
g(\x)+g(\x^\ast)=1\text{ for all }\x \in X.
\end{align} 
Let $\x_0,\x_1,\x_2,\dots$ be the Markov process on $\X$ that goes from $\x_t$ to $(\x_t,j)$ with probability $g((\x_t,j)),j=0,1$, and let  $d\x$ denote  the infinite product of normalized counting measure on $\{0,1\}$. 
Then  \eqref{eq: prob condition a} is equivalent to the condition that 
\begin{align}
\mathbb{P}_{g}(\x_t \to \boldsymbol{0} \text{ or }\boldsymbol{1}\vert\x_0)=1\text{ for $d\x$-almost every 
$\x_0 \in X$.} 
\label{eq: g condition a} 
\end{align}

Before we describe the structure of what Gundy \cite{Gun10} refers to as 
\enquote{inaccessible} invariant sets  we discuss an example %subshifts of finite type studied in 
from \cite[Sect. 13]{Gun07}. For $n\geq 2$, let $K(n)$ be the set of all $\x\in \X$ that do not contain a string (or word) of $n$ consecutive zeros, or a string of $n$ consecutive ones. Then 
\begin{align}\label{eq: K(2)}
K(2) =\{(\dots{},1,0,1,0) \text{ and } (\dots{},0,1,0,1)\}=\tau^{-1}(B),
\end{align} 
where  $B=\{1/3,2/3\}$ is the set discussed  in the introduction. For every $n\geq 2$, we have that $K(n)$ is a closed shift-invariant proper subset of $\X$ (a subshift). Such sets have measure  zero by the ergodicity of the shift  with respect to $d\x$. 

Suppose that we have defined $g$ so that $g$ is continuous and such that $K:=K(3)$ is  invariant, i.e., $\mathbb{P}_{g}(\x_1 \in K\vert\x_0)=1$ for every $\x_0 \in K$. The last condition is met if and only if $g(\x)=0$ for every  $\x \in \Ke$, where $$\Ke:=\{\x  \in K^c\colon (\dots,x_{-2},x_{-1})\in K\}$$ is the set of \enquote{points of exit} from $K$ \cite{Jon17}. It is possible  to define $g$ in such a way that $g$ has no zeros outside $\Ke$ besides the zeros at $\boldsymbol{0}^\ast$ and $\boldsymbol{1}^\ast$, which are required for $g(\boldsymbol{0})=g(\boldsymbol{1})=1$ \cite{Kri06}. To prevent sample paths from initial points in the complement of $K$ from converging to $K$, we modify $g$ so that  
\begin{align}
\Ue{:=} \{\x \in \X \colon (x_{-2},x_{-1},x_0)=(0,0,0) \text{ or }(1,1,1)\}
\end{align}
is visited infinitely often. (If $\x_t \in \Ue$, then  $\rho(\x_t,K)\geq 2^{-3}$,  so   paths that visit $\Ue$ infinitely often do not converge to $K$.) By Levy's conditional Borel-Cantelli Lemma (see \cite{Che78} or  \cite[Lemma 4.1]{Gun07}), we  have $\x_t \in \Ue$ for infinitely many values of $t \geq 1$, $\mathbb{P}_g(\hspace{0.5 mm}.\hspace{0.5 mm}\vert\x_0)$-a.s., if 
\begin{align}
\sum_{t=0}^\infty \mathbb{P}_g(\x_{t+1}\in \Ue|\x_{t})=+\infty, \hspace{1 mm} \mathbb{P}_g(\hspace{0.5 mm}.\hspace{0.5 mm}\vert\x_0)\text{-a.s.}
\end{align}
The words $(0,0)$ and $(1,1)$  are \enquote{critical} in the sense that if one of these words appear as the initial word in $\x_t$,  then $\Ue$ can be reached in one step. By our assumptions on $g$, the probability to reach 
\begin{align}\label{eq: sum is infinite}
\Uc{:=} \{\x\in \X \colon (x_{-1},x_0)=(0,0) \text{ or }(1,1)\}
\end{align}
 in at most two steps is positive for every $\x_0 \in X$. (If $\x_0\in \Uc$, no steps have to be taken. If $\x_0 \in K$, then either $(\x_0,0)$ and $(\x_0,0,0)$ are both in $K$, or $(\x_0,1)$ and $(\x_0,1,1)$ are both in $K$. Since $g$ is strictly positive on $K$, we can then reach $\Uc$ in two steps. Finally, if $\x_0 \in (K \cup \Uc)^c$, then neither $(\x_0,0)$ nor $(\x_0,1)$ is in $\Ke \cup \{\boldsymbol{0}^\ast,\boldsymbol{1}^\ast\}$, so both transitions have positive probability. Since $(\x_0,1)\in \Uc$ if $(\x_0,0)\notin \Uc$,  we can then reach $\Uc$ in one step.) The strictly positive finite-step transition probability is a continuous function of $\x_0$, so it is bounded away from zero. By the Renewal Theorem, we can find $\beta>0$, not depending on $\x_0 \in X$, such that the   recurrence times $t_1,t_2,\dots{}$ for critical words (i.e., the times when $\x_t \in \Uc$) satisfy $t_j \leq \beta j$, $\mathbb{P}_g(\hspace{0.5 mm}.\hspace{0.5 mm}\vert\x_0)$-a.s. Setting $g=|\log_2\rho(\x,\Ke)|^{-1}$ on $\Ue \backslash \Ke$, with a corresponding modification on $\Ue^\ast \backslash (\Ke)^\ast$, we get 
$$\mathbb{P}_g(\x_{t_j+1}\in \Ue\vert\x_{t_j}) \geq \frac{1}{l+t_j}\geq \frac{1}{l+\beta j},$$
where $l$ is the integer with $\rho(\x_0,K)=2^{-l}$.  (Here we have used the fact that  $\rho(\x_0,K)=2^{-l}$ implies $\rho(\x_t,\Ke) \geq 2^{-(t+l)}$: the initial word  in $\x_t$ of length $t+l$ cannot be the initial word of a point of   $\Ke$ if the initial word in $\x_0$ of length $l$ does not appear in a point of  $K$.) 
Because \eqref{eq: sum is infinite} holds,  $\Ue$ is visited infinitely often. If we set $g \equiv 1$ on a neighbourhood of $\{\boldsymbol{0},\boldsymbol{1}\}$, then $\mathbb{P}_{g}(\x_t \to \boldsymbol{0} \text{ or }\boldsymbol{1}\vert\x_0)$ is positive for every $\x_0 \in \Ue$. We then obtain that    the equality in \eqref{eq: g condition a} holds for all  $\x_0 \in K^c$, hence almost everywhere. 

The above construction  relies (only) on the assumption that $K$ is a subshift of finite type (see \cite[Def. 2.1.1]{LM95}). If $K$ is a  $g$-invariant subshift that is not of finite type, then $g$ must take the value zero at some point of $K$  \cite{Kri06}. (The frontier of $\Ke$ is a non-empty subset of $K$ if $K$ is not of finite type \cite{Jon17}. Since $g$ must vanish on $\overline{\Ke}$ if $g$ is continuous and $K$ is invariant,  we must then have $g(\x)=0$  for certain $\x \in K$.) This may leave us without a lower bound on the probability to encounter a critical word in any number of steps.  However, as long as the zeros of $g$ are contained in $\overline{\Ke} \cup (\Ue)^\ast$,  the set 
$\{\boldsymbol{0},\boldsymbol{1}\}$ remains accessible from any $\x_0 \in K^c$ in the sense that $\mathbb{P}_g(\rho(\x_k,\{\boldsymbol{0},\boldsymbol{1}\}) \leq 2^{-k}\vert\x_0)>0$ for every $k \geq 1$.  
Consider therefore a sequence of (dependent) trials, where trial $n\geq 0$ consists of the attempt to reach 
\begin{align}
U_{0,1}{:=} \{\x \in X\colon (x_{-k},\dots,x_0)= 
(\underbrace{0,0,....,0}_\text{$k+1$ zeros }) \text{ or }(\underbrace{1,1,....,1}_\text{$k+1$ ones }) 
\}
\end{align}
by $k$ consecutive steps towards either $\boldsymbol{0}$ or $\boldsymbol{1}$, depending on whether the initial symbol in $\x_{nk}$ is $0$ or $1$. For $k$ so large that $U_{0,1}$ is disjoint from $K$ and with $g(\x)=|\log_2\rho(\x,\overline{\Ke})|^{-1/k}$ on a neighbourhood  of $\overline{\Ke}$, we  obtain (below), for some $\lambda'>0$ and all $n \geq 1$, that \begin{align}
\mathbb{P}_g(\x_{nk+k} \in U_{0,1}|\x_{nk}) \geq \frac{\lambda'}{l+nk+k}, \text{ where $l=|\log_2\rho(\x_0,K)|$.}
\label{tp}
\end{align}
Setting $g \equiv 1$ on $U_{0,1}$ achieves \eqref{eq: g condition a} since $U_{0,1}$ is visited infinitely often if $\x_0 \in K^c$, again by Borel-Cantelli.  

A construction of the second type is  possible whenever $K \subset X \backslash \{\boldsymbol{0},\boldsymbol{1}\}$ satisfies 
\begin{align}
\overline{\Ke} \cap (\overline{\Ke})^\ast=\emptyset.
\end{align}
This condition is necessary if we  require that $g$  be continuous,  for  $g$-invariance then implies that $g(\x)=0$ for all $\x \in \overline{\Ke}$ (the closure of $\Ke$) and, hence, that $g(\x)=1$ for all $\x \in (\overline{\Ke})^\ast$ (cf. \cite{Kri06,Jon17}). The   construction does not  answer the question from the introduction, however, as it does not provide   a continuous $p(\xi), \xi \in \R$. To answer the question that we started out with, we  return to the unit interval.

 \section{Proving Theorem \ref{thm: main result}} \label{sec: thm1}
 \subsection{Definitions}
When we say that $B\subset (0,1)$ is invariant  under multiplication by 2 (mod 1), we mean that if $B$ is considered as a subset of the  circle $[0,1)$, then $B=\theta(B)$, where $\theta(\xi){:=} 
2\xi$ (mod 1).  

 The map $\xi\mapsto \xi^\ast=\xi+1/2$ (mod 1),  which is unambiguously defined on $[0,1)$, corresponds to the map in \eqref{eq: ast on X}.  We define $\xi^\ast$ for all $\xi \in [0,1]$ by  
\begin{align}\label{def: ast}
\xi^\ast{=}\begin{cases}
\xi+1/2\text{ if } \xi \in [0,1/2],\\
\xi-1/2 \text{ if } \xi \in (1/2,1].
\end{cases}
\end{align}
The first condition in \eqref{eq: qmf condition intro}  then says that
\begin{align}\label{eq: 5th general qmf condition}
{p}(\xi)+{p}(\xi^\ast)=1\text{ for every $\xi \in [0,1]$}. 
\end{align}
For $E\subset [0,1]$, we let $E^\ast=\{\xi^\ast\colon \xi\in E\}$. Finally, the distance between $\xi \in [0,1]$ and $E$ is given by
\begin{align}
\label{def: distance function}
d_{E}(\xi)&=\inf_{\xi'\in E}|\xi-\xi'|.
\end{align}   

\subsection{The structure of invariant sets} 
A  set $B\subset [0,1]$ is invariant for the process \eqref{eq: transitions intro} if and only if $p(\xi)=0$  for every $\xi\in \Be$, where 
\begin{align}
\Be&=\{\xi\in [0,1] \backslash B\colon \xi=\xi'/2+j/2\text{ for some $\xi'\in B, j\in\{0,1\}$}\}. 
\label{eq: Be}
\end{align}
(If we identify 0 and 1,   we can write $\Be=\{\xi \in [0,1) \backslash B\colon  \theta(\xi) \in B\}$.) If this condition holds, then we have $p(\xi)=0$ for every $\xi \in \overline{\Be}$ and, consequently, $p(\xi)=1$ for every $\xi \in (\overline{\Be})^\ast$, so that 
 \begin{align}
\overline{\Be} \cap \overline{(\Be})^\ast=\emptyset.
\label{eq: disjointness condition 0}
\end{align}  
The proof  of Theorem \ref{thm: main result} shows that every closed ${\theta}$-invariant $B\subset (0,1)$ that satisfies \eqref{eq: disjointness condition 0} is invariant for some $p$. 

\subsection{Proof of Theorem \ref{thm: main result}} The proof  of Theorem \ref{thm: main result} goes in two steps. Given a    closed  ${\theta}$-invariant $B\subset (0,1)$ that satisfies \eqref{eq: disjointness condition 0},  we first   construct  $p$ such that  $B$ is invariant. We then  verify that $p$ satisfies   \eqref{eq: prob condition a}. 

\subsection*{Step 1:  Construction}  Our construction relies on the following result. 
\begin{lemma}\label{lemma: N and B} Suppose that $B\subset (0,1)$ is closed and $\theta$-invariant.  

\textbf{(a)} $\{0,1/2\} \cap (B\cup \overline{\Be} \cup (\overline{\Be})^\ast)=\emptyset$. 

\textbf{(b)} If $B$ satisfies \eqref{eq: disjointness condition 0},  there is a closed  $N_e\subset [0,1]$ such that

(i) $N_e$ contains $\overline{\Be}$ is in its interior,   

(ii) $N_e \cap N_e^\ast=\emptyset$,

(iii) $\{0,1/2,1\} \cap (N_e\cup N_e^\ast)=\emptyset$.

\end{lemma} 
\begin{proof}  
\textbf{(a)} That $B$ is ${\theta}$-invariant means that $1/2\notin B$.  So we can find $\varepsilon>0$ such that $B\subset C_\varepsilon:=(\varepsilon,1/2-\varepsilon)\cup (1/2+\varepsilon,1-\varepsilon)$. Then we have $\Be\subset C_\varepsilon$. It follows that $\{0,1/2\} \cap (B\cup \overline{\Be} \cup (\overline{\Be})^\ast)=\emptyset$.  

\textbf{(b)} That $B$  satisfies \eqref{eq: disjointness condition 0} means that we can take  $\delta >0$ so that $|\xi-\xi'|>\delta$ if $\xi \in \overline{\Be}$ and $\xi'\in (\overline{\Be})^\ast$. We can cover  the (compact) set $\overline{\Be}$ by a finite union of closed intervals whose lengths do not exceed $\delta/3$ and that each contain a point of $\overline{\Be}$ in its interior. If we let $N_e$ be such a union, then $N_e$ contains $\overline{\Be}$   in its interior. The set $N_e^\ast$ is a finite union of closed  intervals whose lengths do not exceed $\delta/3$ and that each contain  a point of  $(\overline{\Be})^\ast$   in its interior.  Since $\{0,1/2,1\} \cap (\overline{\Be} \cup (\overline{\Be})^\ast)=\emptyset$,  we have $\{0,1/2,1\} \cap (N_e\cup N_e^\ast)=\emptyset$ if we take the intervals that define $N_e$  sufficiently short.   Our choice of $\delta$ gives  $N_e\cap N_e^\ast=\emptyset$.  
 \end{proof}

Now, given a closed  $\theta$-invariant $B\subset (0,1)$ that  satisfies \eqref{eq: disjointness condition 0}, let  $N_e$ be as in Lemma \ref{lemma: N and B}. Choose $\varepsilon>0$  so small   that $N_e\cup (N_e)^\ast$ is disjoint from $$\N0{:=}[0,\varepsilon]\cup [1-\varepsilon,1].$$ 
Then  $N_e\cup (N_e)^\ast$ is also disjoint from  $$N_{\frac{1}{2}}{:=}[1/2-\varepsilon,1/2+\varepsilon]=\N0^\ast.$$
 Fix a positive integer  $k$  with $2^{-k}<\varepsilon$. (This choice of $k$ will ensure that  starting from any $\xi_0 \in [0,1]$ and using the transitions  \eqref{eq: transitions intro},  we  can reach $\N0$ by $k$ consecutive  steps towards 0, or by $k$ consecutive  steps towards 1.)  Define 
\begin{align}\label{def: p}
p(\xi)=\begin{cases}
|\log_2(d_{\Be}(\xi))|^{-1/k} \mbox{ if } \xi \in N_e \backslash \overline{\Be},\\
0 \text{ if } \xi \in \overline{\Be},\\
0 \text{ if } \xi \in N_{\frac{1}{2}}.
\end{cases}
\end{align}
For $\xi\in (N_e)^\ast\cup \N0$, let $p(\xi)=1-p(\xi^\ast)$. Now $p$ is defined on 
\[
N{:=} N_e \cup (N_e)^\ast \cup N_{\frac{1}{2}}\cup \N0
\]
 and the equality in  \eqref{eq: 5th general qmf condition} holds if $\xi\in N$. Extend $p$ to $[0,1/2]\backslash N$  continuously in such a way that   $0<p(\xi)<1$ for all $\xi \in [0,1/2]\backslash N$.  If we   set  $p(\xi)=1-p(\xi^\ast)$ for $\xi \in [1/2,1]\backslash N$ and extend periodically,  then $p$ is a QMF function with  $\{\xi\in [0,1]\colon p(\xi)=0\}= \overline{\Be}\cup N_{\frac{1}{2}}$. In particular,  $B$ is invariant for $p$. 
\newline

\noindent \textbf{Step 2: Condition  \eqref{eq: prob condition a}.} The verification of  \eqref{eq: prob condition a} uses  \eqref{eq: inequality Be t steps} below, which gives an estimate on the speed at which sample paths can approach $\Be$. We take the sample space for the process \eqref{eq: transitions intro} to be the set $\{0,1\}^\mathbb{N}$  of all   binary sequences $x^+=(x_1,x_2,\dots{})$, each $x_i \in \{0,1\}$, and define the sample path $\xi_{t}(x^+)$  from a fixed $\xi_0 \in [0,1]$    recursively via
\begin{align}
\xi_{t}=\xi_{t-1}/2+x_{t}/2, t \geq 1.
\label{eq: recursion1}
\end{align}
Note that the estimates in \eqref{eq: inequality B t steps} and \eqref{eq: inequality Be t steps} below do not involve $p$. 
 \begin{lemma}\label{lemma: xit approaches B and Be} Let $B\subset (0,1)$ be closed and ${\theta}$-invariant,  and let $\xi_0 \in B^c$. There is a constant $\alpha=\alpha(\xi_0)>0$ such  that for any sample path $\xi_t=\xi_t(x^+),t\geq 0,$ from $\xi_0$,   
\begin{align}\label{eq: inequality B t steps}
d_{B}(\xi_t) \geq  \alpha 2^{-t}\text{ for all }t\geq 0,\\
\label{eq: inequality Be t steps}
d_{\Be}(\xi_t) \geq \alpha 2^{-t}\text{ for all }t\geq 1.
\end{align} 
\end{lemma}

\begin{proof}  Since $1/2\notin B\cup \overline{\Be}$ (Lemma \ref{lemma: N and B}(a)), we can pick $\delta\in (0,1)$ such that $|\xi-1/2|>\delta$ for all  $\xi\in B\cup \overline{\Be}$. Let  $\xi_0 \in B^c$ be given and consider the sample path $\xi_t=\xi_t(x^+)$   defined by  $x^+\in \{0,1\}^\mathbb{N}$ and the recursion \eqref{eq: recursion1}. We   show that \eqref{eq: inequality B t steps} holds with  $$\alpha{:=}\min(\delta,d_B(\xi_0)).$$  
This choice of $\alpha>0$ gives  $d_{B}(\xi_t) \geq \alpha 2^{-t}$ if $t=0$. So it is enough to show that  $d_{B}(\xi_{t-1}) \geq \alpha 2^{-(t-1)}$ implies $d_{B}(\xi_{t}) \geq \alpha 2^{-t}$ for all $t\geq 1$. Suppose therefore that $d_{B}(\xi_{t-1}) \geq \alpha 2^{-(t-1)}$, where $t\geq 1$. To estimate $d_{B}(\xi_{t})$, fix an arbitrary $\xi \in B$. Since $B$ is $\theta$-invariant,   we can write $\xi=\xi'/2+j/2$ with $\xi'\in B$ and $j\in\{0,1\}$.   If $j\neq x_{t}^+$, then $|\xi_{t}-\xi|\geq \delta$. (This is because $B  \subset C_\delta:=(0,1/2-\delta)\cup (1/2+\delta,1)$.)  In this case we immediately get $|\xi_{t}-\xi|\geq \alpha 2^{-t}$ from  the definition of $\alpha$.  If  $x_{t}^+=j$, then 
\begin{align*} 
|\xi_{t}-\xi|=|\frac{\xi_{t-1}}{2}+\frac{x_t^+}{2}-(\frac{\xi'}{2}+\frac{j}{2})|=|\xi_{t-1}-\xi'|/2\geq d_B(\xi_{t-1})/2.
\end{align*} 
Using that $d_B(\xi_{t-1})\geq \alpha 2^{-(t-1)}$, we   get $|\xi_{t}-\xi|\geq \alpha 2^{-t}$. Since $\xi \in B $ was arbitrary,   $d_{B}(\xi_{t}) \geq \alpha 2^{-t}$.
 
Now we prove  \eqref{eq: inequality Be t steps}.  Let $t\geq 1$. 
To estimate $d_{\Be}(\xi_t)$, let $\xi \in \Be$, so that (by  the definition of $\Be$) $\xi=\xi'/2+j/2$ for some $\xi'\in B$ and $j\in\{0,1\}$. If $j\neq x_t^+$, then  $|\xi_t-\xi|\geq \delta$. (This follows from that $\Be  \subset C_\delta$.) Thus $|\xi_t-\xi|\geq  \alpha 2^{-t}$ if $j\neq x_t^+$. If  $x_t^+=j$,  then
\begin{align*} 
|\xi_{t}-\xi|=|\frac{\xi_{t-1}}{2}+\frac{x_t^+}{2}-(\frac{\xi'}{2}+\frac{j}{2})|=|\xi_{t-1}-\xi'|/2\geq d_B(\xi_{t-1})/2,
\end{align*}  
so  $|\xi_t-\xi| \geq   \alpha 2^{-t}$. Since $\xi \in \Be$ was arbitrary, $d_{\Be}(\xi_t)\geq \alpha 2^{-t}$.  
\end{proof}

Let $B, N_e, \N0, N_{1/2}$, and $p$ be as in Step 1. Since $B$ has measure zero, we are done if we can show that $\mathbb{P}_p(\xi_t \to 0 \text{ or }1|\xi_0)=1$ for every $\xi_0 \in B^c$.  
Let $\xi_0$ be any point of  $B^c$. That $p\equiv 1$ on $\N0$ means that  if a sample path from $\xi_0$  reaches $\N0$, it goes to $0$ (if it reaches $[0,\varepsilon]$) or $1$ (if it reaches   $[1-\varepsilon,1]$). So it suffices to show that $\xi_t \in \N0$ for some $t$,  $\mathbb{P}_p(\hspace{0.5 mm}.\hspace{0.5 mm}\vert\xi_0)$-a.s. By  Borel-Cantelli,  we will have $\xi_t \in \N0$ for some (in fact infinitely many) $t$ if  
\begin{align}
\sum_{n=0}^\infty \mathbb{P}_p(\xi_{nk+k} \in \N0\vert \xi_{nk})=+\infty, \hspace{1 mm}  \mathbb{P}_p(\hspace{0.5 mm}.\hspace{0.5 mm}\vert \xi_0)\text{-a.s.} 
\label{c0}
\end{align}
We verify \eqref{c0} by showing that there is a constant $\lambda \in (0,1)$ and $a>0$ such that  
\begin{align}\label{eq: lambda inequality}
\mathbb{P}_p(\xi_{nk+k} \in \N0|\xi_{nk}) \geq   \frac{\lambda}{a+nk+k}
\end{align}  
for all $n \geq 1$, $\mathbb{P}_p(\hspace{0.5 mm}.\hspace{0.5 mm}\vert \xi_0)$-a.s.  

\noindent \emph{Case 1: $\xi_{nk} \leq 1/2$.} Then $\xi_{nk}/2^k \in \N0$ by our choice of $k$, so   
\begin{align*} 
\mathbb{P}_p(\xi_{nk+k} \in \N0\vert\xi_{nk}) &\geq \mathbb{P}_p(\xi_{nk+k}=\xi_{nk}/2^k\vert\xi_{nk})
=\prod_{i=1}^kp(\xi_{nk}/2^i).
\end{align*}
That $\xi_{nk} \leq 1/2$ implies  that $\xi_{nk}/2^i \leq 1/4$ for all $i \geq 1$. Since $\{\xi\colon p(\xi)=0\}=\overline{\Be} \cup N_{\frac{1}{2}}$ and $\overline{\Be}$ is in the interior of $N_e$, we can find $c \in (0,1)$ such that $p(\xi)\geq c$ for all  $\xi\in ([0,1] \backslash N_e)  \cap ([0,1/4]\cup [3/4,1])$. Then we have $$\prod_{i=1}^kp(\xi_{nk}/2^i) \geq c^k \text{ if }\xi_{nk}/2^i \notin  N_e  \text{ for }i=1,\dots,k.$$ To verify that  \eqref{eq: lambda inequality} holds, we need a lower bound on  $\prod_{i=1}^kp(\xi_{nk}/2^i)$ for the case when $\xi_{nk}/2^i \in N_e$ for at least one  $i \in \{1,\dots,k\}$. By Lemma \ref{lemma: xit approaches B and Be}, we can choose $\alpha_0>0$ so that $d_{\Be}(\xi_t(x^+))\geq  2^{-t-\alpha_0}$  for every sample path $\xi_t(x^+)$ from  $\xi_0$. For $i \in \{1,\dots,k\}$, take $x^+\in \{0,1\}^\mathbb{N}$ so that $\xi_{nk+i}(x^+)=\xi_{nk}/2^i$. (The first $nk$ entries of $x^+$ define the itinerary from $\xi_0$ to $\xi_{nk}$, and $x^+_{nk+j}=0$ for $j=1,\dots,i$.) If $\xi_{nk}/2^i \in N_e$,   the definition \eqref{def: p} of $p$ together with  \eqref{eq: inequality Be t steps} gives 
\begin{align*} 
p(\xi_{nk}/2^i) = |\log_2(d_{\Be}(\xi_{nk}/2^i))|^{-1/k}&= |\log_2(d_{\Be}(\xi_{nk+i}(x^+)))|^{-1/k} \\
&\geq |\log_2(2^{-(nk+i+\alpha_0)})|^{-1/k}\\
&= (\frac{1}{\alpha_0+nk+i})^{1/k}.
\end{align*}
 Letting $i_1,\dots,i_m$ be the  $m$ ($m\leq k$) 
 integers $i$ with  $\xi_{nk}/2^i  \in N_e$, 
\begin{align*}
 \prod_{i=1}^kp(\xi_{nk}/2^i) &\geq c^{k-m}\prod_{j=1}^{m}\frac{1}{(\alpha_0+nk+i_j)^{1/k}} \geq c^{k-m}\cdot\frac{1}{\alpha_0+nk+k}.
\end{align*}
 This shows that \eqref{eq: lambda inequality} holds with $a=\alpha_0$ and $\lambda=c^{k}$. 

\noindent \emph{Case 2: $\xi_{nk} > 1/2$.} If $\xi_{nk} > 1/2$, then $\N0$ can be reached by $k$ consecutive steps to the right:  $\xi_{nk+i}=\xi_{nk+i-1}/2+1/2$ for $i=1,\dots,k$. We then have $\xi_{nk+i} \geq 3/4$ and the above $c$ bounds $p(\xi_{nk+i})$  when $\xi_{nk+i} \in   ([0,1] \backslash N_e)  \cap ([0,1/4]\cup [3/4,1])$. Lemma \ref{lemma: xit approaches B and Be} and the argument in  Case 1 gives $p(\xi_{nk+i}) \geq (\alpha_0+nk+i)^{-1/k}$ when $\xi_{nk+i} \in N_e$. This means that \eqref{eq: lambda inequality} again holds with $a=\alpha_0$ and $\lambda=c^{k}$. 

Since $n\geq 1$ was arbitrary,  \eqref{eq: lambda inequality}  holds for all $n\geq 1$ with $a=\alpha_0(\xi_0)$ and $\lambda=c^{k}$. Hence, \eqref{c0} is satisfied.

\section{Concluding remarks}
\label{sect: conclude}
A QMF function $p(\xi), \xi \in \R$, generates a scaling function for a MRA if and only if  the infinite product
\begin{align}
\hat{\Phi}_p(\xi){:=}\prod_{j=1}^\infty p(\xi/2^j), \xi \in \mathbb{R},
\end{align}
satisfies  (see   \cite{Gun10} or \cite{HW96})
\begin{align}
&\text{   }\sum_{k \in \mathbb{Z}}\hat{\Phi}_p(\xi+k)=1 \mbox{ for a.e. $\xi \in [0,1]$,}\label{eq: condition (a)}
\\&\text{  } \lim_{j \rightarrow \infty} \hat{\Phi}_p(2^{-j}\xi)=1 \text{ for a.e. } \xi \in \mathbb{R}. \label{eq: condition (b)}
\end{align}
That \eqref{eq: condition (b)} holds for the $p$ that we constructed in the previous section follows from that this $p\equiv 1$ on an open interval containing 0. That the equality in \eqref{eq: condition (a)} holds almost everywhere  for this $p$ follows from that for every $p$ and \emph{every} $\xi_0\in [0,1]$, we have (see \cite{Gun10})  
\begin{align}
\sum_{k \in \mathbb{Z}} \hat{\Phi}_p(\xi_0+k)=\mathbb{P}_p(\xi_t \to 0 \text{ or }1|\xi_0). 
\label{eq: basic conclusion}
\end{align}

The discovery of continuous $p$ for which $\sum_{k \in \mathbb{Z}} \hat{\Phi}_p(\xi+k)=1$ fails on a set of measure zero was made in \cite{DGH00}. The notion of an inaccessible invariant set comes from \cite{Gun07}, where the example from \cite{DGH00} is included in a class of  such invariant sets obtained from subshifts of finite type. In this paper we have described their structure completely.

%%%%%%%%%%%%%%%%%%%%%%%%%%%%%%%%%%%%%%%%%%%%%%%%%%%%%%%%%%%%%%%%%%%
%%                                                               %%
%% Use the two commands below for producing your bibliography    %%
%% with bibtex, then comment again the commands and include the  %%
%% content of the .bbl file in this file below the commands.     %%
%%                                                               %%
%%%%%%%%%%%%%%%%%%%%%%%%%%%%%%%%%%%%%%%%%%%%%%%%%%%%%%%%%%%%%%%%%%%
\bibliographystyle{plain}
\bibliography{qmf_references}
\end{document}